\documentclass[12pt]{amsart}
\usepackage[english]{babel}
\usepackage{amsmath}
\usepackage{amsthm}
\usepackage{amssymb}
\usepackage{mathrsfs}
\usepackage{enumerate}
\usepackage[notcite, final, notref]{showkeys}
\usepackage{dsfont}
\usepackage{color}
\usepackage{shadow}
\newtheorem{theorem}{Theorem}[section]
\newtheorem{problem}[theorem]{Problem}
\newtheorem{lemma}[theorem]{Lemma}
\newtheorem{proposition}[theorem]{Proposition}
\newtheorem{corollary}[theorem]{Corollary}
\newtheorem{definition}[theorem]{Definition}

\theoremstyle{definition}
\newtheorem{remark}[theorem]{Remark}
\newtheorem{example}[theorem]{Example}

\newcommand{\CC}{\mathbb{C}}

\newcommand{\C}{\Phi}
\newcommand{\R}{\mathbb R}
\newcommand{\D}{\Psi}

\newcommand{\e}{\varepsilon}

\title[On Pseudo-Spectral Factorization
over Complex Numbers and Quaternions]{On Pseudo-Spectral Factorization
over the Complex Numbers and Quaternions}

\author[D. Alpay]{Daniel Alpay}
\address{(DA)
Faculty of Mathematics, Physics, and Computation\\
Schmid College of Science and Technology\\
Chapman University\\
One University Drive
Orange, California 92866\\
USA}
\email{alpay@chapman.edu}

\author[F. Colombo]{Fabrizio Colombo}
\address{(FC) Politecnico di
Milano\\Dipartimento di Matematica\\Via E. Bonardi, 9\\20133 Milano,
Italy}
\email{fabrizio.colombo@polimi.it}

\author[I. Lewkowicz]{Izchak Lewkowicz}
\address{(IL) School of electrical and computer engineering
Ben-Gurion University of the Negev\\ P.O.B. 653\\ Beer-Sheva, 84105\\
Israel}
\email{izchak@ee.bgu.ac.il}

\author[I. Sabadini]{Irene Sabadini}
\address{(IS) Politecnico di
Milano\\Dipartimento di Matematica\\Via E. Bonardi, 9\\20133 Milano\\Italy}
\email{irene.sabadini@polimi.it}

\begin{document}
\maketitle
\begin{abstract}
This paper is a continuation of the research of our previous work \cite{MR3904447} and considers quaternionic generalized Carath\'eodory functions and the related family of generalized positive
functions.  It is addressed to a wide audience which includes researchers in complex and hypercomplex analysis, in the theory of linear systems, but also electric engineers. For this reason it includes some results on generalized Carath\'eodory functions and their factorization in the classic complex case which might be of independent interest. An important new result is a pseudo-spectral factorization and we  also discuss some interpolation problems in the class of quaternionic generalized positive functions.

\end{abstract}



\noindent {\em }
\date{today}
\tableofcontents
\section{Introduction}
\setcounter{equation}{0}
This paper is a continuation of our previous work \cite{MR3904447}. We study in the quaternionic
setting generalized Carath\'eodory functions and the related (but slightly different) family of generalized positive
functions. The notion of rational function was recently extended to the case of slice-hyperholomorphic functions, see \cite{MR3568012,acs1,MR3127378,acs_trend_1,MR3192300,zbMATH06658818},
and part of the classical theory has already been considered in
this setting. In \cite{MR3904447} we considered the generalized positive real lemma, which gives a characterization of generalized positive functions in terms of their minimal realizations.
Building on \cite{MR3904447}, we here prove a factorization result in the class of even positive functions, denoted by $\mathcal{GPE}$, in the setting of rational slice-hyperholomorphic functions.
\smallskip

In scalar terminology, positive functions, analytic on the right half plane $\mathbb C_r$, analytically map it to its closure ($\mathbb C_r\bigcup i\mathbb R$).
Generalized positive functions are functions of bounded type in $\mathbb C_r$ and map (in the sense of boundary values) $i\mathbb R$ to $\mathbb C_r$:
\begin{equation}
\label{realpos}
{\rm Re}\,\C(iy)\ge 0
\end{equation}
for all real $y$ where the boundary value $\C(iy)$ is defined.
In electrical engineering and mathematical analysis these functions are analytically extended to
the open left half plane using an integral formula; see \cite{MR48:904,Ca1,Ca2,richards,MR0132753}. When the function we start with is scalar-valued and
real (in the sense that $\overline{\C(\overline{z})}=\C(z)$), the extension is then odd  as is usually the case in electrical engineering. In the present work, we consider rational functions and will not extend them in such a
way, but rather consider them as meromorphic functions in $\mathbb C$. We will denote by $\mathcal{GP}$ the space of (in general matrix-valued) generalized positive functions.\smallskip

We set

\begin{equation}
\C^\sharp(z)=(\C(-\overline{z}))^*.
\label{notthesharpest}
\end{equation}
It is useful to remark that
\begin{equation}
z+\overline{z}=0\,\,\,\Longrightarrow\,\,\, \C(z)^\sharp=\C(z)^*.
\end{equation}
Before presenting our result, we recall a number of definitions. A rational $\mathbb C^{n\times n}$-valued function $\C$ is called odd if
\begin{equation}
\label{oddd}
\C^\sharp(z)=-\C(z)
\end{equation}
and even if
\begin{equation}
\label{evennn}
\C^\sharp(z)=\C(z).
\end{equation}
In particular, in the scalar case, when the coefficients and the variable are real, we get to the usual definitions of odd and even functions.
When $\C$ is even, condition \eqref{realpos} becomes
\[
\C(iy)\ge 0
\]
i.e. $\C$ takes positive (and in particular self-adjoint) values on the imaginary line ($\C\in\mathcal{GPE}$).
Let now $L$ be a ${\mathbb C}^{n\times m}$-valued rational function (we allow $n\not =m$). Clearly
\begin{equation}
\label{rotoro}
\C(z)=L^\sharp(z)L(z)
\end{equation}
is a rational generalized positive function. It is natural to pose the converse question: Does every even generalized positive definite function admits a
factorization of the form \eqref{rotoro}? An answer to this question is given in \cite{MR0132753} using polynomial methods and
\cite[Theorem 10.2, p. 199]{MR2663312} using state space theory. We also refer to \cite{MR3207131}.
The result as presented in \cite{MR2663312} is recalled in Section \ref{sec2!!!!}, see Theorem
\ref{ot200}, and we only outline it in the present introduction. We refer to Section \ref{sec2!!!!} for the definition of a pseudo-spectral factorization.
Assuming the rational generalized positive function $\C$ analytic at infinity and such that $\C(\infty)>0$ (and in particular, $\det\C(z)\not\equiv 0$), $\C$ admits
uniquely defined right and left pseudo-spectral factorizations
\[
\C(z)=L_+^\sharp(z) L_+(z)=L_-^\sharp(z) L_-(z).
\]
Moreover explicit formulas are available for the pseudo-spectral factors $L_+$ and $L_-$ in terms of a minimal realization of $\C$. See
formulas \eqref{corona123}-\eqref{mainform2}. These formulas play a key role in our proof.\\

Before turning to the quaternionic setting we need to recall the following: A $\mathbb C^{n\times n}$-valued rational function $\C$ is a generalized positive
function if and only if the kernel
\begin{equation}
\label{KPHI}
K_\C(z,w)=\frac{\C(z)+\C(w)^*}{z+\overline{w}}
\end{equation}
has a finite number, say $\kappa$, of negative squares in the domain of analyticity of $\C$ in $\mathbb C_r$. See Section \ref{GCF} for a discussion and for
the definition of a kernel having a finite number of negative squares.\\

The purpose of this paper is on one hand to prove the counterpart of the above factorization and realization
result for even rational slice-hyperholomorphic functions, which are generalized positive in a suitable sense and on the other hand to prove some interpolation results.
In the sequel the symbol
$\star$ denotes the star product of (left) hyperholomorphic functions;
we send the reader to Section \ref{sec4} for more information on the terminology and the notation. In particular, $K_\Phi(p,q)$ defined by \eqref{CCCC} is now the
quaternionic counterpart of the kernel \eqref{KPHI}. Extending \eqref{notthesharpest} to the quaternionic setting we define:

\begin{definition}
Let $\C(p)=\sum_{k=0}^\infty p^k\C_k$. We set
\begin{equation}
\C(p)^\sharp=\sum_{k=0}^\infty (-p)^k\C_k^*.
\label{notsosharp}
\end{equation}
The function $\C$ is called even if
\begin{equation}
\label{PR1}
\C^\sharp(p)=\C(p).
\end{equation}
\end{definition}

We remark that, in opposition to the complex case, we will in general have $\C^\sharp(p)\not=\C(p)^*$ when $p+\overline{p}=0$.\smallskip

To define generalized positive functions in the quaternionic case we need to resort to kernels with a finite number of negative squares.

\begin{definition}
The $\mathbb H^{n\times n}$-valued slice-hyperholomorphic rational function is called generalized positive even if it is even and if the kernel $K_\C(p,q)$ defined by
\begin{equation}
\label{CCCC}
(\C(p)+\C(q)^*)\star(p+\overline{q})^{-\star}
\end{equation}
has a finite number of negative squares in the open half-space, from which are removed the spheres of poles of $\C$.
\end{definition}

We will use the notation $\C\in\mathcal {GPE}(\mathbb H)$ for quaternionic generalized positive even functions.

\begin{theorem}
Let $\C$ be a $\mathbb H^{n\times n}$-valued slice-hyperholomorphic rational function belonging to $\mathcal {GPE}(\mathbb H)$,
slice hyperpolomorphic at infinity with value $I_n$ there, i.e.
\[
\lim_{p\rightarrow\infty}\C(p)=I_n,
\]
and  with minimal realization
\begin{equation}
\C(p)=I_n+C\star(pI_N-A)^{-\star}\star B.
\label{finfin}
\end{equation}
Then, there exist $\mathbb H^{n\times n}$-valued slice-hyperholomorphic rational functions $L_+$ and $L_-$,  respectively, without poles and zeros in the open right
half-space and in the open left half-space, uniquely determined by the condition $L_\pm(\infty)=I_n$, and such that
\begin{equation}
\label{ronron}
\C(p)=L_+^\sharp (p)\star L_+(p)=L_-^\sharp (p)\star L_-(p),
\end{equation}
where the $\sharp$ is defined in \eqref{notsosharp}.
\label{even}
\end{theorem}

To prove this theorem we use the map $\chi$ (see Definition \ref{chi} below) which allows to consider the complex-valued setting, and use an analytic extension argument from \cite{MR3904447}
and the formulas from \cite[Theorem 10.2, p. 199]{MR2663312}.\\

We note that the family of generalized positive functions forms a convex invertible cone (CIC), both in the classical and quaternionic setting (one needs to consider
the $\star$-product in the latter case), i.e. a convex cone for which invertible elements are still in the cone; see \cite{CohenLew1997a,CohenLew1997b,LewRodYar2005}.

\begin{remark}
{\rm This paper is written for more than one audience, in particular researchers from the fields of theory of linear systems, hypercomplex analysis and electrical engineering. We did not try to be self-complete
(that would be impossible in the setting of a paper), but we have recalled a number of facts which may be well-known to one of the aimed audiences, and not to the other ones. These various groups may have different terminologies, for instance what we call in this paper positive is sometimes called semi-positive.
We hope it is clear to all potential readers.
The theory of rational slice-hyperholomorphic functions is relatively recent, and we review for the benefit of the readers from
hypercomplex analysis some known results in the classical case, such that Proposition \ref{WW*}.}
\end{remark}

The paper consists of 6 sections besides the introduction, and we now review its content.\\
In section 2 besides to provide some preliminary  notions on matrix valued rational functions, their realizations, we discuss various important facts like the generalized positive lemma and pseudo-spectral factorizations.
In section 3 we give a number of interpolation results which are consequences of the results in section 2 and we also show various examples. Section 4 deals with generalized Carath\'eodory functions and their factorization, also discussing some examples. Section 5 moves to the quaternionic case and contains some preliminary facts, whereas section 6 contains the proof of Theorem 1.3 in which we prove a factorization result in the class of even rational slice-hyperholomorphic functions and two corollaries. Finally, in section 7 we discuss some interpolation problems in the class of quaternionic generalized positive functions.

\section{The rational case}
\setcounter{equation}{0}
\label{sec2!!!!}
In the discussion, and also in later sections, we will use kernels with a finite number of negative squares, first introduced by Krein; see \cite{krein-1959}.
We refer to the paper \cite[\S 9]{stewart} for an historical survey  of the notion.

\begin{definition}
\label{co}
Let $K(z,w)$ be a $\mathbb C^{n\times n}$-valued function (also called {\sl kernel} defined for $z,w$ in some set $\Omega$. We say that $K(z,w)$ has a finite number,
say $\kappa$, of negative squares in $\Omega$ if it is Hermitian:
\begin{equation}
\label{krein}
K(z,w)=K(w,z)^*,\quad\forall z,w\in\Omega,
\end{equation}
for every choice of $N\in \mathbb N$, $z_1,\ldots, z_N\in\Omega$ and $c_1,\ldots, c_N\in\mathbb C^n$, the $N\times N$ matrix with $(j,k)$ entry
$c_j^*K(z_j,z_k)c_k$ (which is Hermitian in view of \eqref{krein}) has at most $\kappa$ strictly negative eigenvalues, and exactly $\kappa$ strictly negative
eigenvalues for some choice of $N,z_1,\ldots, z_N$ and $c_1,\ldots, c_N$.
\end{definition}

\begin{remark}{\rm
When $\kappa=0$ the notion reduces to the notion of positive definite function (or kernel). Since the spectral theorem holds for quaternionic
Hermitian matrices (see e.g. \cite{MR97h:15020}), the definition still makes sense in the quaternionic setting.}
\label{negsq}
\end{remark}

We begin by recalling the concept of state space realization. Let
$\Phi(z)$ be a $n\times n$-valued rational function analytic at
infinity, i.e.  $\lim\limits_{z~\rightarrow~\infty}\Phi(z)~$ exists.  Then, $\Phi(z)$ admits a state space realization
\begin{equation}
\label{minPhi}
\begin{matrix}
\Phi(z)=D+C(zI_N-A)^{-1}B&~&~&
R_{\Phi}:=\left(
\begin{array}{c|c}
A&~B\\
\hline
C&~D
\end{array}\right)\end{matrix}
\end{equation}
with $~A\in\CC^{N\times N}$,
\mbox{$B, C^*\in\CC^{N\times n}$} ~and $~D\in\CC^{n\times n}$, namely,
\mbox{$R_{\Phi}\in\CC^{(N+n)\times(N+n)}$}. If $N$ is the smallest possible,
it is called the McMillan degree of $\Phi(z)~$ and
the realization is called {\em minimal}.\\


We note that $\Phi(\infty)=D$. A minimal realization is unique up to a uniquely defined and invertible
similarity matrix $S\in\mathbb C^{N\times N}$, meaning that two minimal realizations
$\C(z)= D+C_1(zI_N-A_1)^{-1}B_1=D+C_2(zI_N-A_2)^{-1}B_2$ are related by
\begin{equation}
\begin{pmatrix}
 S &0 \\
 0 &I_n\end{pmatrix}
\begin{pmatrix}A_1&B_1\\ C_1&D\end{pmatrix} =\begin{pmatrix}A_2&B_2\\ C_2&D\end{pmatrix} \begin{pmatrix}
 S &0 \\
 0 &I_n\end{pmatrix}.
\end{equation}
The realization is minimal if and only the pair $(C,A)$ is observable and the pair $(A,B)$ is controllable,
meaning respectively
\begin{equation}
\bigcap\limits_{u=0}^{N-1}\ker CA^u=\left\{0\right\}\quad {\rm and}\quad
\bigcup\limits_{u=0}^{N-1} {\rm ran}\, A^uB=\mathbb C^N.
\label{CABA}
\end{equation}

Assume now $\C$ rational, analytic at infinity, and with minimal realization \eqref{minPhi}.
The positive real lemma, see \cite{Anderson_Moore,faurre,MR525380} and the generalized positive
real lemma, see \cite{AlpLew2011,DDGK,MR2008274,zbMATH06743412}, characterize these classes in terms of the
given realization.

\begin{theorem}
Let $\C$ be a $\mathbb C^{n\times n}$-valued rational function analytic at infinity, with minimal realization
\eqref{minPhi}. Then, $\C$ is a generalized positive function if only if there exists an invertible Hermitian
matrix $H\in\mathbb C^{N\times N}$ such that
\begin{equation}
\label{opera4!4!6!7!}
\begin{pmatrix}H&0\\0&I_{n}\end{pmatrix}
\begin{pmatrix}A&B\\ C&D\end{pmatrix}+
\begin{pmatrix}A&B\\ C&D\end{pmatrix}^*
\begin{pmatrix}H&0\\0&I_{n}\end{pmatrix}\ge 0.
\end{equation}
\end{theorem}

We remark that the matrix $H$ is not uniquely determined. Following \cite{faurre} we denote by
\[
\begin{pmatrix}Q&S\\ S^*&R\end{pmatrix}\ge 0
\]
the left handside of \eqref{opera4!4!6!7!}. Then
\begin{align}
HA+A^*H&=Q\\
HB+C^*&=S\\
D+D^*&=R.
\end{align}
We then have for $z,w\in\mathbb C$,
\begin{equation}
\begin{split}
K_\Phi(z,w)&=\frac{\Phi(z)+\Phi(w)^*}{z+\overline{w}}\\
&=-C (zI_N-A)^{-1}H^{-1}(\overline{w}I_N-A^*)^{-1}C^*+\\
&\hspace{5mm}+\frac{\left(
\begin{pmatrix}
C(zI_N-A)^{-1}H^{-1}&I_n\end{pmatrix}\begin{pmatrix}Q&S\\ S^*&R\end{pmatrix}
\begin{pmatrix}H^{-1}(\overline{w}I_N-A^*)^{-1}  C^*\\I_n\end{pmatrix}\right)}{ z+\overline{w}}.
\end{split}
\label{opera_bastille}
\end{equation}

Equation \eqref{opera_bastille} has the following important corollary (these formulas can be found e.g.
in \cite[p. 129]{MR525380}):

\begin{corollary}\mbox{}\\
$(a)$  Let  $\C$ be a $\mathbb C^{n\times n}$-valued rational function, analytic at infinity, and with a
real positive part on $i\mathbb R$. Then the associated kernel $K_\C(z,w)$ has a finite number of negative
squares in $\mathbb C_r$. Conversely, if the matrix-valued rational function $\C$ is such that the kernel
$K_\C$ has a finite number of negative squares
in $\mathbb C_r\cap\Omega(\C)$, it belongs to ${\mathcal GP}$.\\
$(b)$ Denoting by $\Omega(\C)$ the set of points of analyticity of $\C$, we have:
\begin{equation}
\C(z)+\C^\sharp(z)=
\begin{pmatrix}
C(zI_N-A)^{-1}H^{-1}&I_n\end{pmatrix}\begin{pmatrix}Q&S\\ S^*&R\end{pmatrix}
\begin{pmatrix}H^{-1}(-zI_N-A^*)^{-1} C^*\\I_n\end{pmatrix},\quad z\in\Omega(\C).
\end{equation}
In particular, $\C$ does not satisfy in general the symmetry condition \eqref{oddd}.
\label{cy789}
\end{corollary}

\begin{proof} $(a)$ In \eqref{opera_bastille}, the kernel
\[
-C (zI_N-A)^{-1}H^{-1}(\overline{w}I_N-A^*)^{-1}C^*
\]
is finite dimensional and so has both a finite number of negative squares and a finite number of negative squares. Moreover since the matrix
\[
\begin{pmatrix}Q&S\\ S^*&R\end{pmatrix}\ge 0
\]
the kernel
\[
\frac{
\begin{pmatrix}C(zI_N-A)^{-1}H^{-1}&I_n\end{pmatrix}\begin{pmatrix}Q&S\\ S^*&R\end{pmatrix}
\begin{pmatrix}H^{-1}(\overline{w}I_N-A^*)^{-1}  C^*\\I_n\end{pmatrix}}{ z+\overline{w}}\ge 0.
\]
Thus $K_\C(z,w)$ has a finite number of negative squares.\smallskip

Conversely, assume thay $K_\C(z,w)$ has a finite number of negative squares in $\mathbb C_r\cap\Omega(\C)$. A quick proof
(but which leaves the realm of rational functions) is to take the representation \eqref{place_de_la_nation} below  for $\C$. We can apply the factorization result of
\cite{MR1771251,MR1736921,MR1902953,MR1971748}. Writing
$\C(z)=V^\sharp(z)\C_0(z)V(z)$ for $z=iy+\e$, $\e>0$ we get
\[
\C(iy+\e)+\C(iy+\e)^*=V(iy-\e)\C_0(iy+\e)V(iy+\e)+V(iy+\e)^*\C_0(iy+\e)^*V(iy-\e),
\]
with limit as $\e\longrightarrow 0$
\[
\C(iy)+\C(iy)^*=V(iy)(\C_0(iy)+(\C_0(iy))^*)V(iy)^*\ge 0.
\]
Note that the functions are rational and so the limits are well defined. One could also have used a Cayley transform and reduce the case to that of a
generalized Schur function, and use the
Krein-Langer factorization, since the latter shows in particular that such a function takes contractive boundary values. More precisely, recall that a
$\mathbb C^{n\times n}$-valued function meromorphic in $\mathbb C_r$ is called a generalized Schur function if the kernel
\begin{equation}
\label{groningen1987}
K_S(z,w)=\frac{I_n-S(z)S(w)^*}{z+\overline{w}}
\end{equation}
has a finite number of negative squares for $z,w\in\Omega(S)$. Then, $S=B_0^{-1}S_0$, where $B_0$ is a
$\mathbb C^{n\times n}$-valued finite Blaschke product, and where $S_0$ is a $\mathbb C^{n\times n}$-valued
function analytic and contractive in $\mathbb C_r$. This factorization result, due to Krein and Langer, is
proved in \cite{kl1} in the setting of the open unit disk. In \cite{MR2002664} Bolotnikov and Rodman give a
proof in the setting of meromorphic functions, which can be read as is for rational functions.\smallskip

Formula \eqref{groningen1987} shows that a generalized Schur function takes (non-tangential) contractive
boundary values on $i\mathbb R$. The result for a generalized Carath\'eodory function is obtained, as
mentioned earlier, via Cayley transform.\smallskip

The claims in $(b)$ are obtained by multiplying both sides of \eqref{opera_bastille} by $z+\overline{w}$ and setting $z=w$.
\end{proof}

The following proposition will be used in the sequel, and in particular in the statement of Theorem
\ref{ot200}; it appears in \cite{ag,MR682303}.

\begin{proposition}
\label{WW*}
Let $\C$ be a $\mathbb C^{n\times n}$-valued function, analytic at infinity and let $\C(z)=D+C(zI_N-A)^{-1}B$
be a minimal realization of $\C$. Then, it holds that
\begin{equation}
\label{yasmina}
\C^\sharp(z)=\C(z), \quad z\in\Omega(\C)
\end{equation}
if and only if there exists a skew-Hermitian matrix $H$ such that
\begin{equation}
\label{dimanche}
\begin{pmatrix}
 H &0 \\
 0 &I_n\end{pmatrix}
\begin{pmatrix}A&B\\ C&D\end{pmatrix} =\begin{pmatrix}-A^*&C^*\\ -B^*&D^*\end{pmatrix} \begin{pmatrix}
 H &0 \\
 0 &I_n\end{pmatrix}=\begin{pmatrix}A^*&C^*\\ B^*&D^*\end{pmatrix}
\begin{pmatrix}
 -H &0 \\
 0 &I_n\end{pmatrix}
\end{equation}
Furthermore, $H$ is invertible, and uniquely determined from the given realization
\end{proposition}

\begin{proof}
A minimal realization of $\C^\sharp(z)=(\C(-\overline{z}))^*$ is given by
\[
(\C(-\overline{z}))^*  =D^*-B^*(zI_N+A^*)^{-1}C^*.
\]
Thus, by uniqueness up to similarity of a minimal realization, equation \eqref{yasmina}
is equivalent to the existence of an invertible uniquely defined matrix $H$ such that \eqref{dimanche} holds.
Taking adjoint on both sides of \eqref{dimanche} we obtain
\[
\begin{pmatrix}
 -H^* &0 \\
 0 &I_n\end{pmatrix}
\begin{pmatrix}A&B\\ C&D\end{pmatrix} =\begin{pmatrix}A^*&C^*\\ B^*&D^*\end{pmatrix}\begin{pmatrix}
 H^* &0 \\
 0 &I_n\end{pmatrix}
\]
The uniqueness of $H$ forces $H=-H^*$. The above equalities also imply that $HB=C^*$.
\end{proof}

\begin{remark}{\rm We note that the previous result {\sl does not} characterize the case where
$\C(iy)\ge 0$ on the imaginary line. We also note that \eqref{dimanche} can be rewritten as
\begin{eqnarray}
HA&=&-A^*H\\
HB&=&C^*
\end{eqnarray}
together with $D=D^*$.}
\end{remark}

We next address the state space realization of $\mathcal{GPE}$ functions. To this end we first consider a
 realization of a product of a pair of rational functions (series or cascade connection
in electrical engineering terminology). See e.g. \cite[Subsection 8.3.3]{MR569473}.

\begin{proposition}\label{Pn:series}
Given $l\times q$ and $q\times r$-valued rational functions
$L_{\alpha}(z)$, $L_{\beta}(z)$ admitting state space realization
\begin{equation}\label{eq:OriginalRealization}
\begin{matrix}
R_{L_{\alpha}}=\left(
\begin{array}{c|c}
A_{\alpha}&B_{\alpha}\\
\hline
C_{\alpha}&D_{\alpha}
\end{array}\right)
&&&
R_{L_{\beta}}=\left(
\begin{array}{c|c}
A_{\beta}&B_{\beta}\\
\hline
C_{\beta}&D_{\beta}
\end{array}\right),
\end{matrix}
\end{equation}
where $A_{\alpha}$, $A_{\beta}$ are in $\CC^{N_{\alpha}\times N_{\alpha}}$ and $\CC^{N_{\beta}\times N_{\beta}}$, respectively. A realization of
$L_{\alpha}(z)L_{\beta}(z)$ is given by
\begin{equation}\label{eq:RealizationProduct}
R_{L_{\alpha}L_{\beta}}=
\left(\begin{array}{cc|c}
A_{\alpha}&B_{\alpha}C_{\beta}
&B_{\alpha}D_{\beta}\\
0&A_{\beta}
&B_{\beta}\\ \hline
C_{\alpha}&~D_{\alpha}C_{\beta}
&D_{\alpha}D_{\beta}
\end{array}\right).
\end{equation}
In the special case that
\[
L_{\beta}(z)=L_{\alpha}^{\#}(z),
\]
one obtains a realization of a $\mathcal{GPE}$ function,
\begin{equation}\label{eq:R1}
R_{L_{\alpha}L_{\alpha}^{\sharp}}=
\left(\begin{array}{cc|c}
A_{\alpha}&B_{\alpha}B_{\alpha}^*
&B_{\alpha}D_{\alpha}^*\\
0&-A_{\alpha}^*&-C_{\alpha}^*\\
\hline
C_{\alpha}&D_{\alpha}B_{\alpha}^*
&D_{\alpha}D_{\alpha}^*
\end{array}\right).
\end{equation}
\end{proposition}

Recall that even when the original realizations of $L_{\alpha}$
and of $L_{\beta}$ are minimal, the resulting realization of the
product $L_{\alpha}L_{\beta}$ is not necessarily minimal.
In particular, if $L_{\beta}=L_{\alpha}^{-1}$, the McMillan degree
of the product is zero. In contrast, for the realization of
a $\mathcal{GPE}$ function in \eqref{eq:R1} we have the
following:

\noindent
\begin{theorem}
The realization $R_{L_{\alpha}L_{\alpha}^{\sharp}}$ in \eqref{eq:R1}
is minimal, if and only if the realization $R_{L_{\alpha}}$ in
\eqref{eq:OriginalRealization} is minimal.\smallskip

If the realization in \eqref{eq:R1} is not minimal, it is neither controllable nor observable.
\end{theorem}
\vskip 0.2cm

\begin{proof}
Following the Popov-Belevich-Hautus eigenvector tests for controllability
and observability, see e.g. \cite[Subsection 2.4.3]{MR569473}, (adapted
to the realization array notation) the realization
$R_{L_{\alpha}{L_{\alpha}^{\sharp}}}$ is not observable if there exist
$v_1,~v_2\in\mathbb{C}^{N_{\alpha}}$ (not both zero) so that for some
$\lambda\in\mathbb{C}$,
\begin{equation}\label{eq:Unboservable}
\left(\begin{matrix}
Av_1+BB^*v_2\\
-A^*v_2\\
Cv_1+DB^*v_2
\end{matrix}\right)
=\lambda
\left(\begin{matrix}
v_1\\
v_2\\
0
\end{matrix}\right).
\end{equation}
Similarly, the realization in \eqref{eq:R1} is not
controllable if there exist
$u_1,~u_2\in\mathbb{C}^{N_{\alpha}}$ (not both zero) so that for some
$\mu\in\mathbb{C}$,
\[
\mu
\left(\begin{matrix}
u_1^*&u_2^*&0
\end{matrix}\right)
=
\left(\begin{matrix}
u_1^*A&&u_1^*BB^*-u_2^*A^*&&u_1^*BD^*-u_2^*C^*
\end{matrix}\right).
\]
Multiplying by
$-\left(\begin{smallmatrix}0&I_n&0\\
I_n&0&0\\0&0&I_m\end{smallmatrix}\right)$
from the right and taking $(~)^*$, this is
equivalent to
\begin{equation}\label{eq:Uncotrollable}
\left(\begin{matrix}
Au_2
-BB^*u_1
\\
-A^*u_1
\\
Cu_2
-DB^*u_1
\end{matrix}\right)
=-\mu^*
\left(\begin{matrix}
u_2\\
u_1\\
0
\end{matrix}\right).
\end{equation}
Substituting in \eqref{eq:Uncotrollable}
\[
u_2=v_1\quad u_1=v_2\quad-\mu^*=\lambda,
\]
one obtains \eqref{eq:Unboservable}
\end{proof}

We refer to  \cite[p. 199]{MR2663312} for Definition \ref{dassin}. In preparation to the statement, we also
recall the following (see \cite[p. 175]{MR2663312}):
Given $\C$ with minimal realization \eqref{minPhi}, one defines the local McMillan $\delta(\C,z_0)$ degree of
$\C$ at the point $z_0$ to be the algebraic multiplicity of $z_0$ as an eigenvalue of $A$. The uniqueness of
a minimal realization up to a similarity matrix  ensures that the definition is independent of the given
minimal realization. A factorization $\C=\C_1\C_2$ of $\C$ into two $\mathbb C^{n\times n}$-valued rational
functions is a locally minimal factorization at the given point means that the
local McMillan degrees of $\C_1$ and $\C_2$ add up at this point:
\[
\delta(\C,z_0)=\delta(\C_1,z_0)+\delta(\C_2,z_0).
\]

\begin{definition}
\label{dassin}
The factorization $\C(z)=L^\sharp(z)L(z)$ is called right pseudo-spectral with respect to $i\mathbb R$
if $L$ has no poles or zeros in the open left half-plane and the factorization is locally minimal at
each point on the imaginary axis. Replacing the open left half-plane by the open right half-plane gives
the corresponding notion of left pseudo-spectral factorization.
\end{definition}

For completeness, and to ease the reading of the proof of Theorem \ref{even}, we recall here (with some differences of notation) the statement of the result from
\cite[Theorem 10.2, p. 199]{MR2663312} which we will use. For the existence and uniqueness of the
Hermitian matrix $H$ in the statement, see Proposition \ref{WW*}. We note that the uniqueness of
the pseudo-spectral factor is a key ingredient in the proof of Theorem \ref{even}.

\begin{theorem}
\label{ot200}
Let $\C(z)$ be a $\mathbb C^{n\times n}$-valued function analytic at infinity and with minimal realization $\C(z)=D+C(zI_N-A)^{-1}B$. Assume that $D>0$ and that $\C(iy)\ge 0$ for $y\in\mathbb R$
where $\C(iy)$ is defined. Then $\C$ admits right and left pseudo-spectral factorizations, obtained as follows. Let $H$ be the unique Hermitian matrix defined by $HA=A^*H$ and $HB=C^*$. Then,
there exist
$A$-invariant subspaces $\mathcal M_+$ and $\mathcal M_-$, and $A^\times$-invariant subspaces $\mathcal M^\times_+$ and $\mathcal M_-^\times$, uniquely determined by the following five conditions:
\begin{enumerate}
\item [{\rm (i)}] $\mathcal M_+$ contains the spectral subspace of $A$ associated with the part of $\sigma(A)$ lying in the open right half plane, and $\sigma(A|_{\mathcal M_+})\subset\left\{
    z | {\rm Re}\, z\ge 0 \right\}$.

\item [{\rm (ii)}] $\mathcal M_-$ contains the spectral subspace of $A$ associated with the part of $\sigma(A)$ lying
in the open right half plane, and $\sigma(A|_{\mathcal M_+}) \subset \left\{z |{\rm Re}\, z\le  0\right\}$.

\item [{\rm (iii)}] $\mathcal M_+^\times$ contains the spectral subspace of $A^\times$ associated with the part of $\sigma(A^\times)$ lying in the open right half plane, and
  $\sigma(A^\times|_{\mathcal M^\times_+})\subset\left\{
    z | {\rm Re}\, z\ge 0 \right\}$.

\item [{\rm (iv)}]  $\mathcal M^\times_-$ contains the spectral subspace of $A^\times$ associated with the part of $\sigma(A^\times)$ lying
in the open left half plane, and $\sigma(A^\times|_{\mathcal M_-}) \subset \left\{z |{\rm Re}\, z\le  0\right\}$.

\item [{\rm (v)}] $H\mathcal M_+=\mathcal M_+^{\perp}$,\quad $H\mathcal M_-=\mathcal M_-^{\perp}$,\quad$H\mathcal M_+^\times={\mathcal M_+^\times}^{\perp}$,\quad $H\mathcal M_-^\times={\mathcal M_-^\times}^{\perp}$.
\end{enumerate}
The subspaces in question also satisfy the matching conditions (where $\stackrel{\cdot}{+}$ denotes a direct sum)
\begin{equation}
\label{mathcond}
\mathbb C^N= \mathcal M_-\stackrel{\cdot}{+}\mathcal M_+^\times\quad  and\quad \mathbb C^N=\mathcal M_+\stackrel{\cdot}{+}\mathcal M_-^\times
\end{equation}
Let $\Pi_+$ denote the projection along $\mathcal M_-$ onto $\mathcal M_+^\times$, let $\Pi_-$ denote the projection along $\mathcal M_+$ onto $\mathcal M_-^\times$, and let
\begin{eqnarray}
\label{corona123}
L_+(z)&=&D^{1/2}+D^{-1/2}C\Pi_+(zI_N-A)^{-1}B,\\
L_-(z)&=&D^{1/2}+D^{-1/2}C\Pi_-(zI_N-A)^{-1}B.
\label{mainform2}
\end{eqnarray}
Then $L_+$ and $L_-$ are right and left pseudo-spectral factors with respect to the imaginary line, with corresponding right and left pseudo-spectral factorizations
\begin{equation}
\C(z)=L_+^\sharp(z)L_+(z)\quad and\quad \C(z)=L_-^\sharp(z)L_-(z).
\label{spectralfacto}
\end{equation}
These factors are uniquely determined by the fact that they have the value $D^{1/2}$ at infinity.
\end{theorem}

We now discuss some consequences of Theorem \ref{ot200}.

\begin{corollary}
\label{adelle}
If $\C$ is a polynomial so are the pseudo-spectral factors.
\end{corollary}

\begin{proof}
In a minimal realization $A$ is then nilpotent, and the formulas for $L_\pm$ then give also
polynomials.
\end{proof}

For the following corollary, see also \cite[Proposition 5.2, p. 3961]{MR3034510}.

\begin{corollary}
\label{cor567}
Let $L_1,\ldots, L_U$ be rational $\mathbb C^{n\times n}$-valued functions analytic at infinity and assume that
\begin{equation}
\sum_{u=1}^U(L_u(\infty))^*L_u(\infty)>0.
\end{equation}
Then, there exist $\mathbb C^{n\times n}$-valued rational functions $L_\pm(z)$ analytic at infinity, with no poles and zeros in $\mathbb C_r$ and $\mathbb C_\ell$
(the left open half-plane), respectively, such that
\begin{equation}\label{225}
\sum_{u=1}^UL_u^\sharp(z)L_u(z)=L^\sharp(z) L(z).
\end{equation}
\end{corollary}
\noindent
Note the condition in Theorems 1.3 and 2.8 the assumption that the matrix
$\lim\limits_{z~\rightarrow~\infty}\Phi(z)$ is positive
definite (or even $I_n$), was needed to simplify the
treatment, but it is neither a prerequisite to factorization
nor to realization.
\vskip 0.2cm

\noindent
For example $\Phi(z)=-\frac{1}{z^2}$ is a $\mathcal{GPE}$
function vanishing at infinity. It admits the factorization
in \eqref{225}
with $L=\frac{1}{z}$ and a minimal state space realization
\[
R_{\Phi}=\left(\begin{array}{cc|c}
0&1&0\\0&0&1\\
\hline
1&0&0
\end{array}\right),
\]
which is of the form \eqref{eq:R1}.
\vskip 0.2cm

\noindent
We now have the following corollary to Theorem  \ref{ot200}:

\begin{corollary}
\label{cycy}
Let $\C\in\mathcal{GPE}$ and analytic at infinity. Then there exist factorizations $\C(z)=L_+^\sharp(z)L_+(z)=
L_-^\sharp(z)L_-(z)$, where the poles and zeros of $L_+$ (resp. $L_-$) are in
the closed left half-plane (resp. the closed right half-plane). When $\C$ is a polynomials so are the factors $L_\pm(z).$
\end{corollary}

\begin{proof}
When $\C(\infty)>0$ this is just the previous theorem. Assuming $\C(\infty)$ degenerate, we apply Theorem \ref{ot200} to $\C_\e(z)=\e I_n+\C(z)$, to obtain a family of pseudo-spectral factors
$L_{+,\e}(z)$ satisfying
\begin{equation}
\label{monique}
\e I_n+\C(z)=L_{+,\e}^\sharp(z)L_{+,\e}(z),\quad \e>0.
\end{equation}
We note from formula \eqref{corona123} that
\begin{equation}
L_{+,\e}(z)=\sqrt{\e}I_n+C_\e(zI_N-A)^{-1}B,\quad{\rm where}\quad C_{\e}=\frac{1}{\sqrt{\e}}C\Pi_{+,\e}
\end{equation}
where $\Pi_{+,\e}$ is the projection corresponding to the spaces $\mathcal M_{-,\e}$ and $\mathcal M_{+\e}^\times$ built from $\e I_n+\C(z)$ as in the theorem.
Since the pair $(A,B)$ is controllable (see \eqref{CABA}), we can take $N$ points $y_1,\ldots, y_N$ where $\C(iy_j)$ is well defined and such that $\mathbb C^N$ is
spanned by the columns of the matrices $(y_jI_N-A)^{-1}B$, $j=1,\ldots, N$, i.e.
\begin{equation}
\label{rivka}
\mathbb C^N=\mbox{\text{\rm linear span}}\left\{{\rm ran}\, (y_jI_N-A)^{-1}B,\, j=1,\ldots, N\right\}.
\end{equation}
We restrict $\e\in [0,1]$. We have
\begin{equation}
\label{barbara}
L_{+,\e}(iy_j)^*L_{+,\e}(iy_j)=\e I_n+\C(iy_j)\le I_n+\C(iy_j),\quad \e\in[0,1],\quad j=1,,\ldots, N.
\end{equation}
Thus
\begin{equation}
L_{+,\e}(iy_j)^*L_{+,\e}(iy_j)=\e I_n+\C(iy_j)\le M,\quad \e\in[0,1],\quad j=1,,\ldots, N,
\end{equation}
with $M=I_n+\sum_{k=1}^N \C(iy_k)$. So the $N$ matrices $C_\e(y_jI_N-A)^{-1}B$ are uniformly bounded in norm. By taking converging subsequences we can assume that the limits
\[
\lim_{\e\rightarrow 0} C_{\e}(y_jI_N-A)^{-1}B =H_j,\quad j=1,\ldots N
\]
exist. In view of the full rank hypothesis \eqref{rivka}, this defines in a unique way $X\in\mathbb C^{n\times N}$ such that
\[
  X\begin{pmatrix}(y_1I_N-A)^{-1}B&(y_2I_N-A)^{-1}B&\cdots&(y_NI_N-A)^{-1}B\end{pmatrix}=\begin{pmatrix}H_1&H_2&\cdots &H_N\end{pmatrix},
\]
and $X=\lim_{\e\rightarrow 0}C_{\e}$ because of the full rank hypothesis. This conclude the proof of the existence of $L_+$. The claim for $L_-$ is proved in the
same way, and the claim on polynomials follows from the formulas for the factors, as in Corollary \ref{adelle}.
\end{proof}

\begin{corollary}
\label{cor567890}
Let $L_1,\ldots, L_U$ be rational $\mathbb C^{n\times n}$-valued functions analytic at infinity.
Then, there exist $\mathbb C^{n\times n}$-valued rational functions $L_\pm(z)$ analytic at infinity, with no poles and zeros in $\mathbb C_r$ (resp. in $\mathbb C_\ell$) and such that
\begin{equation}
\label{factor56789}
\sum_{u=1}^UL_u^\sharp(z)L_u(z)=L_+^\sharp(z) L_+(z)=L_-^\sharp(z) L_-(z).
\end{equation}
When $L_1,\ldots, L_U$ are polynomial so are the factors $L_\pm(z).$
\end{corollary}

We note that the factors $L_\pm$ need not be invertible in the preceding two corollaries.

\section{Interpolation}
\setcounter{equation}{0}
\label{INTER}
We give a number of interpolation results which are corollaries of the previous discussion.

\begin{proposition}
Let $w_1,\ldots, w_N\in\mathbb C$ such that $w_u+\overline{w_v}\not=0$ for all $u,v\in\left\{1,\ldots, N\right\}$ (and in particular the points belong to $\mathbb C\setminus i\mathbb R$)
and let $\xi_1,\eta_1,\ldots, \xi_N,\eta_N\in\mathbb C^n$, with $\xi_u\not=0$, $u=1,\ldots, N$.
There exists a $\mathbb C^{n\times n}$-valued rational function $L$, which can be chosen to
be  a polynomial and such that
\begin{equation}
\label{factor5678}
L^\sharp(w_i)L(w_i)\xi_i=\eta_i,\quad i=1,\ldots, N.
\end{equation}
\label{cro123321}
\end{proposition}

\begin{proof}
For every $u\in\left\{1,\ldots, N\right\}$ build a $\mathbb C^{n\times n}$-valued polynomial $P_u(z)$ such that
\begin{equation}
\begin{split}
P_u(w_j)\xi_j&=0,\quad j\not=u,\\
P_u(w_u)\xi_u&=\eta_u,\\
P_u(-\overline{w_u})&=I_n.
\end{split}
\end{equation}
It suffices to take
\[
\C(z)=\sum_{u=1}^NP_u^\sharp(z)P_u(z).
\]
$\C$ satisfies the interpolation conditions and is in ${\mathcal{GPE}}$. So it can be written in the factored form as in \eqref{factor56789}.
\end{proof}

\begin{remark}
{\rm The polynomial $P_u^\sharp P_u$ is the counterpart of the classical Lagrange interpolation polynomial, when one requires positivity on the imaginary axis.}
\end{remark}

When one point, say $w_1$, belongs to the imaginary line we have in particular the condition
\[
L(w_1)^*L(w_1)\xi_1=\eta_1.
\]
Thus, to allow the case in which there are imaginary points  we need
the following simple, but here crucial, fact:

\begin{lemma}
Let $\xi,\eta\in\mathbb C^n$. There exists a positive matrix $A$ such that $A\xi=\eta$ if and only if $\xi^*\eta> 0$ or $\eta=0$.
\end{lemma}

\begin{proof}
The condition is necessary since
\[
0\le \xi^*A\xi=\xi^*\eta.
\]
To study the converse we first consider the case $\xi^*\eta=0$. We then have $\xi^*A\xi=0$ and since $A\ge 0$,
$A\xi=0$ and hence $\eta=0$. We can take $A=I-\frac{\xi\xi^*}{\xi^*\xi}$.
If $\xi^*\eta>0$, the matrix
\[
A=\frac{1}{\eta^*\xi}\eta\eta^*
\]
answers the question.
\end{proof}

\begin{corollary}
Let $w_1,\ldots, w_N\in\mathbb C$ and let $\xi_1,\eta_1,\ldots, \xi_N,\eta_N\in\mathbb C^n$, with $\xi_u\not=0$, $u=1,\ldots, N$.
There exists a $\mathbb C^{n\times n}$-valued rational function $L$, which can be chosen to
be  a polynomial and such that
\begin{equation}
\label{factor5678}
L^\sharp(w_i)L(w_i)\xi_i=\eta_i,\quad i=1,\ldots, N
\end{equation}
if and only if $\xi_u^*\eta_u\ge 0$ for each $u$ such that $w_u\in i\mathbb R$.
\end{corollary}

\begin{proof}
For $u\in\left\{1,\ldots, N\right\}$ with no $w_v$ such that $w_u+\overline{w_v}=0$ we build $P_u$ as in Corollary \ref{cro123321}. Assume now $w_u\in i\mathbb R$ or $w_u,w_v$ with $u\not=v$ and
such that $w_u+\overline{w_v}=0$. We have in particular
\[
L(w_u)^*L(w_u)\xi_u=\eta_u
\]
and, if $u\not=v$,
\[
L(w_v)^*L(w_v)\xi_v=\eta_v
\]
and so the condition $\xi_u^*\eta_u\ge 0$ is indeed necessary.
Let $A_u$ be a positive matrix such that $A_u\xi_u=\eta_u$. We build a polynomial $P_u$ such that
\[
\begin{split}
P_u(w_j)\xi_j&=0,\quad\hspace{2mm} j\not=u,\\
P_u(-\overline{w_u})&=I_n,\\
P_u(w_u)&=A_u^{1/2},
\end{split}
\]
and possibly similarly for $w_v$.
Let
\[
\C(z)=\sum_{\substack{u=1\\ {\exists v\,\,{\rm s.t.}}\\ w_u+\overline{w_v}=0}}^N A_u^{1/2}P_u^\sharp(z)P_u(z)A_u^{1/2}+
\sum_{\substack{u=1\\ w_u+\overline{w_v}\not=0,\forall v}}^N P_u^\sharp(z)P_u(z).
\]
Then, $\C$ solves the interpolation conditions and can be expressed in the factorized form using Corollary \ref{cycy}.
\end{proof}

The above approach was developed in \cite{MR3034510}; another approach, using the fact that within the family of {\em even} functions the $\mathcal{GPE}$ polynomials form a convex cone, was developed in \cite{MR3223890}.
The idea is to build an interpolating polynomial, which takes hermitian, but not necessarily positive values, on the imaginary axis, and perturb it via a generalized positive function vanishing at the
interpolation points. The approach in \cite{MR3034510} is extended in the present paper to the quaternionic setting. Extending the second approach does not seem possible.
We illustrate the second method in the following example.
\begin{example}
Consider interpolation with $\mathcal{GPE}$ polynomials, with nodes and image points given by
\[
\begin{array}{c|c|c|c}
&1&1+i&1-i\\
  \hline
&1&2+8i&2-8i
\end{array}
\]
To find an {\em even} interpolating polynomial we need to add the
constraints,
\[
\begin{array}{c|c|c|c|c|c|c}
&1&1+i&1-i&-1&-1-i&-1+i\\
\hline
&1&2+8i&2-8i&1&2+8i&2-8i
\end{array}
\]
The respective Vandermonde equation is
\[
\left(\begin{smallmatrix}
1&&1   &&1  &&1    &&1 &&1\\~\\
1&&1+i &&2i &&-2+2i&&-4&&-4-4i\\~\\
1&&1-i &&-2i&&-2-2i&&-4&&-4+4i\\~\\
1&&-1  &&~1 &&-1   &&~1&&-1    \\~\\
1&&-1-i&& 2i&&2-2i &&-4&&4+4i\\~\\
1&&-1+i&&-2i&&2+2i &&-4&&4-4i
\end{smallmatrix}\right)
\left(\begin{smallmatrix}
a_o\\~\\a_1\\~\\a_2\\~\\a_3\\~\\a_5\\~\\a_6
\end{smallmatrix}\right)
\left(\begin{smallmatrix}
~~1\\~\\2+8i\\~\\2-8i\\~\\~~1\\~\\2+8i\\~\\2-8i
\end{smallmatrix}\right)
\quad\Longrightarrow\quad
\left(\begin{smallmatrix}
a_o\\~\\a_1\\~\\a_2\\~\\a_3\\~\\a_5\\~\\a_6
\end{smallmatrix}\right)
=
\left(\begin{smallmatrix}
-2\\~\\~~0\\~\\~~4\\~\\~~0\\~\\-1\\~\\~~0
\end{smallmatrix}\right)
\]
Thus, a minimal degree interpolating {\em even} polynomial is
\[
p_2(z)=-z^4+4z^2-2
\]
Next we exploit the fact that the set of {\em even} interpolating
polynomials forms a linear variety.
Take a minimal degree $\mathcal{GPE}$ polynomial vanishing at
the nodes,
\[
\Phi_o(z)=(4+z^4)(1-z^2)=-z^6+z^4-4z^2+4.
\]
Next, for a parameter $\beta\in\mathbb R$ define
\[
\Phi(z):=p_2(z)+\beta\Phi_o(z)=-{\beta}z^6+(\beta-1)z^4+4(1-\beta)z^2+4\beta-2.
\]
$\Phi(z)$ is an {\em even} interpolation polynomial for all $\beta\in\R$.\\

Note that on the imaginary axis $z=iy$, $y\in\R$,

\[
\begin{matrix}
{\Phi(z)}_{|_{z=iy}}&=&
{\beta}y^6+(\beta-1)y^4+4(\beta-1)y^2+4\beta-2
\\~&=&
y^2\left(\beta\left(y^2+\frac{\beta-1}{2\beta}\right)^2
  +(\beta-1)\frac{15\beta+1}{4\beta}\right)+4\beta-2.
\end{matrix}
\]

Thus, using the fact that the set of $\mathcal{GPE}$
functions forms a convex cone, one has that for
$\beta\geq 1$, this in a  $\mathcal{GPE}$ interpolating
polynomial. In particular, for $\beta=1$,
we have the solution, with corresponding left spectral factorization
\[
\Phi(z)=2-z^6=L_-(z)L_-^\sharp(z)
\]
with $L_-(z)=(\sqrt[6]{2}-z)(\sqrt[6]{2}e^{\frac{i\pi}{3}}-z)(\sqrt[6]{2}e^{-\frac{i\pi}{3}}-z)$.\\

Note that although formulated in scalar language,
the above interpolation can be casted in matricial framework.
\end{example}

\section{Generalized Carath\'eodory functions}
\setcounter{equation}{0}
\label{GCF}
Let $\C$ be a $\mathbb C^{n\times n}$-valued function, meromorphic in the open right half-plane $\mathbb C_r$. We call $\C$ a generalized Carath\'eodory
function if the kernel \eqref{KPHI} has a finite number, say $k$, of negative squares (see Definition \ref{co}. We denote by $\mathcal C_k^n$ the set of generalized Carath\'eodory
functions. When $n=1$ we denote $\mathcal C_k^1=\mathcal C_k$, and when $k=0$ we set $\mathcal C_0=\mathcal C$ the set of the so-called Carath\'eodory functions.
\begin{remark}
{\rm
We note that Akhiezer defines in \cite[p. 116]{akhiezer_russian} the Carath\'eodory class as to be the functions analytic in the open unit disk with a positive
real part there, and quotes in particular Herglotz \cite{herglotz} and Riesz \cite{zbMATH02604577} for the result on their integral representation.}
\end{remark}

There are at least two approaches to the theory of generalized Carath\'eodory functions; in the first one
the function $\C$ is extended to $\mathbb C\setminus i\mathbb R$ via
\begin{equation}
\C(z)+\C^\sharp(z)=0.
\label{sym123321}
\end{equation}
Such functions are called odd. We recall that $K_\C$ has the same number of negative squares in $\mathbb C_r$ and in $\mathbb C\setminus i\mathbb R$,
and (see e.g. \cite[(1.1)]{MR1902953}) that $\C$ admits a realization of the form
\begin{equation}
\C(z)=z_0\Gamma^*\Gamma-(z+\overline{z_0})\Gamma^*(I_{\mathcal H}+(z-z_0)(A-zI_{\mathcal H})^{-1})\Gamma,\quad z\in\mathbb C\setminus i\mathbb R,
\end{equation}
where $A$ is a skew-adjoint relation in a Hilbert space $\mathcal H$, $\Gamma\in\mathcal L(\mathbb C^n,\mathcal H)$, and $\Phi_0\in\mathbb C^{n\times n}$ is such
that
\[
\C_0+\C_0^*=(z_0+\overline{z_0})\Gamma^*\Gamma.
\]
Furthermore,
\begin{equation}
\frac{\C(z)+(\C(w))^*}{z+\overline{w}}=\Gamma_w^*\Gamma_z
\end{equation}
with
\[
\Gamma_z=\left(I_{\mathcal H}+(z-z_0)(A-zI_{\mathcal H})^{-1}\right)\Gamma.
\]
Note that, in general, the condition \eqref{sym123321} will prevent an element of $\mathcal C_k^n$ from being meromorphic in the complex plane. In fact,
such functions have an integral representation
generalizing the classical Herglotz representation, see \cite[(4.11) p. 215]{MR47:7504}. The function $\C$ will be meromorphic in the whole complex plane if and
only if an underlying measure appearing in the representation is a jump measure.\smallskip

Generalized Carath\'eodory functions can be characterized in terms of a factorization, which is the counterpart for for generalized Carath\'eodory functions of the Krein-Langer factorization for
generalized Schur functions, and was given in the works \cite{MR1771251,MR1736921,MR1902953,MR1971748},
both in the matrix and operator-valued cases. A constructive way to give the factorization in the scalar case was given later in \cite{a_lew_1}.
The result is:
\begin{theorem}
Let $\C$ be a $\mathbb C^{n\times n}$-valued function analytic in the open right half-plane, extended to the open left-half plane by $\C(z)+\C^\sharp(z)=0$.
Then  the following are equivalent:\\
$(1)$ The kernel $K_\C(z,w)$ has a finite number of negative squares in $\mathbb C\setminus i\mathbb R$.\\
$(2)$ There exist a $\mathbb C^{n\times n}$-valued rational function $V$ and a Carath\'eodory function $\C_0$ such that
\begin{equation}
\label{place_de_la_nation}
\C(z)=V^\sharp(z)\C_0(z)V(z)
\end{equation}
with $V^\sharp(z)$ defined by \eqref{notthesharpest}, i.e. $V^\sharp(z)=V(-\overline{z})^*$.
\label{gare_de_lyon}
\end{theorem}
In Theorem \ref{gare_de_lyon} the function $V$ collects the generalized poles and zeros of $\C$ of negative type; we will recall the definitions in the sequel.
When the kernel is positive definite, $A$ reduces to a constant matrix, assumed to be the identity. The factorization was obtained in \cite{MR1902953} in an iterative way,
by first extracting a factor on the right and on the left corresponding to a pair of a generalized
pole of non-positive type and a generalized zero of non-positive type, when the corresponding
eigenvectors are not orthogonal (the so-called non-orthogonal case). The
general case is then treated in that paper by a pole displacement.\smallskip

\begin{remark}
{\rm When $\C_0(z)\equiv I_n$ in the above (i.e. not extended to $\mathbb C_\ell$ as to satisfy  \eqref{sym123321}), and when local minimality of the factorization is requested at the purely
imaginary points and if $V$ has no poles or zeros in the open left half plane, \eqref{place_de_la_nation} is then a right spectral factorization; see Definition
  \ref{dassin}.}
\end{remark}

The decomposition

\[
\begin{split}
  \frac{V^\sharp(z)\C_0(z)V(z)+V(w)^*\C_0(w)^*(V^{\sharp}(w))^*}{z+\overline{w}}&=V^{\sharp}(z)\frac{\C_0(z)+\C_0(w)^*}{z+\overline{w}}(V^{\sharp}(w))^*+\\
&\hspace{-45mm}
+\begin{pmatrix} I_n&V^\sharp(z)\C_0(z)\end{pmatrix}\begin{pmatrix}0&\frac{V(z)-(V^\sharp(w))^*}{z+\overline{w}}\\
\frac{V(w)^*-(V^\sharp(z))}{z+\overline{w}}&0\end{pmatrix}
\begin{pmatrix} I_n&V^\sharp(w)\C_0(w)\end{pmatrix}^*
\end{split}
\]

shows that the kernel $K_\C(z,w)$ associated with \eqref{place_de_la_nation} has at most ${\rm deg}\,V$ negative squares.\\
\begin{example}
We here illustrate \eqref{place_de_la_nation} through examples.

$(a)$ Consider the function $\Phi=\frac{z^3(1-z)}{(z-i)^2}$ $\C(z)=\frac{z^3(1-z)}{(z-i)^2}$ (see \cite[Example p. 520]{a_lew_1}).
It has a positive real part on the imaginary axis (besides at $z=i$ where it has a pole) and we have
\[
\C(z)=V^\sharp(z)\C_0(z)V(z)
\]
where $\C_0(z)=\frac{z}{z+1}$ is a positive function and $V(z)=\frac{z(z+1)}{z-i}$.

$(b)$ As a two-dimensional example, let us take
\[
\C_0(z)=\begin{pmatrix}z&0\\0&1\end{pmatrix}\quad {\rm and}\quad V(z)=\begin{pmatrix}1&z\\0&1\end{pmatrix}.
\]
Then,
\[
\C(z)=V^\sharp(z)\C_0(z)V(z)=\begin{pmatrix}1&0\\-z&1\end{pmatrix}\begin{pmatrix}z&0\\0&1\end{pmatrix} \begin{pmatrix}1&z\\0&1\end{pmatrix}=\begin{pmatrix}z&z^2\\-z^2&1-z^3\end{pmatrix}
\]
and ${\rm Re}\, \C(ix)=\begin{pmatrix}0&0\\0&1\end{pmatrix}$. Furthermore, for $z,w\in\mathbb C_r$,
\[
\begin{split}
  K_\C(z,w)&=\frac{1}{z+\overline{w}}\begin{pmatrix}z+\overline{w}&z^2-\overline{w}^2\\ \overline{w}^2-z^2&2-z^3-\overline{w}^3\end{pmatrix}\\
  &=\begin{pmatrix}1&z-\overline{w}\\ \overline{w}-z&\frac{2}{z+\overline{w}}-(z^2+\overline{w}^2+z\overline{w})\end{pmatrix}\\
  &=\begin{pmatrix}1&0\\ 0&\frac{2}{z+\overline{w}}\end{pmatrix}+\begin{pmatrix}0&z-\overline{w}\\ \overline{w}-z&-(z^2+\overline{w}^2+z\overline{w})\end{pmatrix}.
  \end{split}
\]
This kernel has one negative square (we will not prove this here).

$(c)$ With $V$ from item $(b)$, consider the $\mathcal{GPE}$ function $V^\sharp V$. Indeed,
\[
V^\sharp(z)V(z)=\begin{pmatrix}1&z\\-z&1-z^2\end{pmatrix}
\]
which for $z=ix$ takes the values
\[
\begin{pmatrix}1&ix\\-ix&1+x^2\end{pmatrix},\quad\mbox{\rm equal to the real part of } V^\sharp(z)V(z).
\]
Finally, for $z,w\in\mathbb C_r$,
\[
\begin{split}
\frac{\begin{pmatrix}1&z\\-z&1-z^2\end{pmatrix}+\begin{pmatrix}1&w\\-w&1-w^2\end{pmatrix}^*}{z+\overline{w}}&=\frac{\begin{pmatrix}2&z-\overline{w}\\ \overline{w}-z&2-z^2-\overline{w}^2\end{pmatrix}}
{z+\overline{w}}\\
&\hspace{-3cm}=\underbrace{2\frac{\begin{pmatrix}1\\ -z\end{pmatrix}\begin{pmatrix}1\\ -w\end{pmatrix}^*
  +\begin{pmatrix}0\\1\end{pmatrix}
\begin{pmatrix}0\\1\end{pmatrix}^*}{z+\overline{w}}}_{\mbox{$K_1(z,w)$}}+\underbrace{\begin{pmatrix}0&1\\1&-(z+\overline{w})\end{pmatrix}}_{\mbox{$K_2(z,w)$}}
\end{split}
\]
which expresses the kernel $K_{V^\sharp V}(z,w)$ as a sum of a positive kernel $K_1(z,w)$ and a kernel $K_2(z,w)$ which has one negative square. The reproducing kernel associated to $K_1(z,w)$
consists of the functions of the form
\[
\begin{pmatrix}1\\ -z\end{pmatrix}h_1(z)+\begin{pmatrix}0\\ 1\end{pmatrix}h_2(z),
\]
where $h_1$a nd $h_2$ run through the Hardy space $\mathbf H_2(\mathbb C_r)$, while the reproducing kernel Pontryagin space associated to $K_2(z,w)$ is spanned by the functions
\[
\begin{pmatrix}1\\ z\end{pmatrix}\quad {\rm and}\quad \begin{pmatrix}0\\ 1\end{pmatrix}.
\]
These spaces have a trivial intersection. Indeed, assume that there exist complex numbers $a,b$ and functions $h_1,h_2\in\mathbf H_2(\mathbb C_r)$ such that
\[
\begin{pmatrix}1\\ -z\end{pmatrix}h_1(z)+\begin{pmatrix}0\\ 1\end{pmatrix}h_2(z)=a\begin{pmatrix}1\\ z\end{pmatrix}+b \begin{pmatrix}0\\ 1\end{pmatrix}.
\]
Then $a=h_1(z)$ and so $a=h_1(z)=0$ and then $b=h_2(z)$ and so $b=h_2(z)=0$. It follows
\[
\mathcal H(K_{V^\sharp V})=\mathcal H(K_1)[+]\mathcal H(K_2),
\]
where $[+]$ denotes a direct and orthogonal sum, and so the kernel $K_{V^\sharp V}$ has exactly one negative square.
\end{example}

In the second approach to the study of generalized Carath\'eodory functions, $\C$ is assumed to be rational, and so in general will not meet \eqref{sym123321}.
The function is then a generalized Carath\'eodory function if and only if it belongs to ${\mathcal GP}$.
This is a known, but non trivial fact; see Corollary \ref{cy789}.\smallskip

The two approaches intersect in a very special class, namely odd rational functions. These were studied using realization theory in particular in \cite{ag}.\\

It is useful to compare the two approaches in a very simple and important example, related to the matrix sign function; see \cite{MR1343800} for the latter. We take $\Phi(z)=\frac{1}{2}$
for ${\rm Re}\, z>0$, and we denote by $\mathbb C_\ell$ the open left half-plane. We have in the first approach
\[
\C(z)=-\frac{1}{2},\quad z\in\mathbb C_\ell,
\]
and the kernel $K_C(z,w)$ is equal to
\[
  K_C(z,w)=\begin{cases}\,\frac{1}{z+\overline{w}},\quad z,w\in\mathbb C_r,\\
    \,\, 0\,\,\,\hspace{8mm} z,w \,\,\mbox{\rm in different half-planes,}\\
    \frac{-1}{z+\overline{w}},\quad z,w\in\mathbb C_\ell,
    \end{cases}
\]
and the associated reproducing kernel Hilbert space consists of functions equal to a function of the Hardy space $\mathbf H_2(\mathbb C_r)$ in $\mathbb C_r$,
and equal to a function of the Hardy space $\mathbf H_2(\mathbb C_\ell)$ in $\mathbb C_\ell$. On the other hand, in the second approach, $K_C$ is defined only in
$\mathbb C_r$, and the associated reproducing kernel Hilbert space consists of functions equal to a function of the Hardy space
$\mathbf H_2(\mathbb C_r)$ in $\mathbb C_r$.\\

In summary:\mbox{}\vspace{1cm}

\shabox{
{\rm
Start from a $\mathbb{C}^{n\times n}$-valued rational function $\Phi$ in ${\rm Re}(z)\geq 0$ and analytic
at infinity.\\
{\bf Case 1:} Extend analytically $\C$ to the whole plane. The positive real lemma characterizes the condition ${\rm Re}\,\C(iy)\ge0$. Then, $K_\C(z,w)$ has a
finite number of negative squares in the whole complex plane, from which poles of $\C$ are removed. But $\C$ need not satisfy \eqref{sym123321}.\\
{\bf Case 2:} Consider $\C$ defined in the open set ${\rm Re}\, z>0$ and extend it to the left open half-plane by \eqref{sym123321}.
The kernel $K_\C(z,w)$ has a finite number of negative squares in $\mathbb C_r$ if and only if it can be written as \eqref{place_de_la_nation}.\\
{\bf Case 3:} Intersection of the two cases, i.e. rational functions satisfying \eqref{sym123321}. This case is studied in \cite{ag} (but no factorization of the kind
\eqref{sym123321} is given there).}
}
\mbox{}\vspace{1cm}

In the quaternionic setting it need not be true that a function $\C$ such that the (counterpart of the) kernel $K_\C$ has a finite number of negative squares, will
have a positive real part on the set of quaternions with real positive part. See Example \ref{567764}.\\

\section{Quaternionic setting: Preliminaries}
\setcounter{equation}{0}
\label{sec4}
In this section we recall some basic notions and results in the quaternionic setting. The skew field of quaternions contains elements of the form $q=x_0+x_1i+x_2j+x_3k$ where $i$, $j$, $k$ satisfy $i^2=j^2=-1$, $ij=-ji=k$, $jk=-kj=i$, $ki=-ik=j$. Given a quaternion $q$ its conjugate is $\bar q=x_0-x_1i-x_2j-x_3k$ and $q\bar q=\bar q q=|q|^2$ where $|q$ denotes the Euclidean norm of $q$.
\\
The set
\[
\mathbb S = \{p = x_1 i + x_2j + x_3k\ {\rm such\ that}\ x_1^2+ x_2^2+ x_3^2 = 1\}
\]
contains purely imaginary quaternions with norm $1$. It is a 2-dimensional sphere in $\mathbb H$
identified with $\mathbb R^4$. We note that an element $I\in\mathbb  S$ satisfies $I^2 = -1$ and thus it behaves like an imaginary unit.
\\
Given any nonreal quaternion $p$, we can write $p=x+Iy$ where $x=x_0$  and $y= x_1 i + x_2j + x_3k/| x_1 i + x_2j + x_3k|$, thus $p\in\mathbb C_I$. We can also define the set
\[
[p]=\{x+Jy\ | \ J\in\mathbb{S}\}
\]
which is a 2-dimensional sphere in $\mathbb{R}^4$ identified with $\mathbb H$. The sphere $[p]$ associated with $p$ can be also seen as the equivalence class of the elements equivalent to $p$ where $q$ is equivalent to $p$ if and only if $q=r^{-1}pr$ for a suitable $r\not=0$.
 If $p=x\in\mathbb R$, $p\in\mathbb{C}_I$ for any $I\in\mathbb S$. It is immediate that the sphere $[x]$ contains only $x$.\\

Let us write a quaternion $q$ in the form $q=(x_0+x_1i)+(x_2+x_3i)j=z_1+z_2j$, where $z_1$, $z_2$ belong to the complex plane $\mathbb C_i$ (we write $\mathbb C$ for simplicity) associated with the imaginary unit $i$. We can define the map
$\chi_i:\ \mathbb H\to \mathbb C^{2\times 2}$ by
\begin{equation}\label{chi}
\chi_i(q)=\begin{pmatrix} z_1&z_2\\
-\overline{z_2}&\overline{z_1}\end{pmatrix}.
\end{equation}
depends on the choice of the imaginary unit $i$, however in the sequel we do not emphasize this dependence and we omit to specify thus writing, for simplicity, $\chi$. The map $\chi$ is extended to matrices in the following
way: If $A=A_1+A_2j$ we set
\begin{equation}
\chi(A)=\begin{pmatrix}A_1&A_2\\ -\overline{A_2}&\overline{A_1}\end{pmatrix}
\end{equation}
and we have
\begin{equation}
\label{Isym}
A\in{\rm ran}\,\chi\,\,\, \iff\,\,\, A=E_n^{-1}\overline{A}E_n
\end{equation}
with
\begin{equation}
\label{EEE111}
E_n=\begin{pmatrix}0&I_n\\-I_n&0\end{pmatrix}.
\end{equation}

We note that symmetries of the type \eqref{Isym} were considered in \cite{MR3223890}, in the setting of polynomial interpolation.

As it is well known, there are various ways of extending to the quaternionic setup the notion of (matrix valued) holomorphic function. In this paper we use the so-called (left) slice hyperholomorphic functions. We will not repeat the basic information, and we refer the reader to \cite{zbMATH06658818} for more details. To our purposes, it is enough to recall the definition (given in a general setting although we use it for the matrix-valued case) and a characterization:
\begin{definition}
Let $\Omega\subseteq \mathbb H$ be an axially symmetric set and let $\mathcal{X}$ be a two sided quaternionic Banach space. A function
$f:\ \Omega\to\mathcal{X}$ of the form
$f(p)=f(x+{I}y)=\alpha (x,y) +{I}\beta (x,y)$ where $\alpha, \beta:
\Omega\to \mathcal{X}$ depend only on $x,y$, are real differentiable, satisfy the Cauchy-Riemann
equations
\begin{equation}\label{CR1}
\begin{cases}
\partial_x \alpha -\partial_y\beta=0\\
\partial_y \alpha
+\partial_x\beta=0,
\end{cases}
\end{equation}
and
\begin{equation}\label{alfabeta1}
\begin{split}
\alpha(x,-y)=\alpha(x,y),\qquad
\beta(x,-y)=-\beta(x,y)
\end{split}
\end{equation}
is said to be (left) slice hyperholomorphic.
\end{definition}
Being in a noncommutative setting, one can also give the definition of right slice hyperholomorphic functions. The definition is as above, but for functions of the form $f(p)=f(x+{I}y)=\alpha (x,y) +\beta (x,y)I$.\\
The next proposition contains a characterization of functions slice hyperholomorphic in a neighborhood of a real point:
\begin{proposition}\label{slicef}
An $\mathbb{H}^{n\times n}$-valued function $f$ is (left) slice hyperholomorphic in a ball $B\subseteq \mathbb{H}$ centered at $x_0\in\mathbb R$ if and only if $f$ is of the form
\[
f(p)=\sum_{m}(p-x_0)^mA_m, \qquad A_m\in \mathbb{H}^{n\times n}, \qquad p\in B.
\]
\end{proposition}
It is immediate that $B$ may coincide with $\mathbb H$ and that polynomials with matrix coefficients written on the right are a particular case of slice hyperholomorphic functions on $\mathbb H$. For right slice hyperholomorphic functions, the series are written with coefficients on the left.
\\
Another consequence of Proposition \ref{slicef} is that pointwise multiplication of two slice hyperholomorphic functions does not belong, in general, to this class of functions. In order to have an inner operation one needs to define a suitable notion of multiplication and in the case of functions slice hyperholomorphic at the origin the operation is described in the following lemma:
\begin{lemma}
\label{jeanette}
Let $F$ and $G$ be $\mathbb H^{n\times n}$-valued functions slice hyperholomorphic in a neighborhood $V$ of the origin, ,and let $p_0\in V$. Let
$F(p)=\sum_{k=0}^\infty p^kF_k$, with $F_k\in\mathbb H^{n\times n}$. We have
\[
(G\star F)(p)=\sum_{k=0}^\infty p^kG(p)F_k
\]
If $G(p_0)=0$ we have
\[
(G\star F)(p_0)=0.
\]
If $G(p_0)=I_n$,
\begin{equation}
\label{zebda}
(G\star F)(p_0)=\sum_{k=0}^\infty p_0^kF_k=F(p_0) .
\end{equation}
\end{lemma}
One peculiar behavior of the $\star$-multiplication is that if $p=p_0$ is a zero of $F(p)$, in general it is not a zero of $G\star F$. This fact has serious consequences for example in interpolation theory.
\\
It is also possible to define the $\star$-inverse of a slice hyperholomorphic function, at the points where it
is  nonzero. For the goals of this paper, it is enough to recall the formula
\[
(p+\overline q)^{-\star}=(|q|^2+2{\rm Re}(q)p+p^2)^{-1}(p+q).
\]

In \cite[p. 1767]{acs-survey}, we considered an example in which the operator of $\star$-multiplication was an isometry on the Hardy space but not contractive. We next consider a similar example, adapted to this framework.

\begin{example}
\label{567764}
Let
\[
\C(p)=\begin{pmatrix}1&i\\ j&ij\end{pmatrix}
\star\begin{pmatrix}p&0\\ 0&1\end{pmatrix}\star \begin{pmatrix}1&i\\ j&ij\end{pmatrix}^*=\begin{pmatrix}p+1&-pj+j\\
  pj-j&p+1\end{pmatrix}.
\]
Then, $\C\in \mathcal C_0^2$ but it does not hold that
\[
{\rm Re}\, \C(p)\ge 0\quad {\rm for}\quad p+\overline{p}=0.
\]
\end{example}
Indeed,
\[
(\C(p)+\C(q)^*)\star(p+\overline{q})^{-\star}=Z
Z^*\ge 0,{\rm with}\quad Z=\begin{pmatrix}1&i\\ j&ij\end{pmatrix},
\]
while, for $p+\overline{p}=0$,
\[
\begin{split}
2{\rm Re}\, \C(p)&=2{\rm Re}\,\begin{pmatrix}p+1&-pj+j\\
  pj-j&p+1\end{pmatrix}\\
&=\begin{pmatrix}2&-pj-j\overline{p}+2j\\pj+j\overline{p}-2j&2
\end{pmatrix}
\end{split}
\]
For $p=tk$ with real $t$ and $k\in\mathbb S$ such that $jk=-kj$, we have
\[
pj+j\overline{p}=2t kj
\]
and the matrix ${\rm Re}\, \C(p)$ is not positive for large enough real $t$.\\

We now present, in the setting of quaternions and in a weaker version, a result from \cite[Proposition 3.3]{bolotnikov2019lagrange}. We adapt the result in the sense that we do not require the mixing
(or compatibility) conditions
ahead of time, but rather require that mixing conditions can be met; this is the solvability of the equations \eqref{tototo123321}. In the present notation, in \cite{bolotnikov2019lagrange},
the matrices $X_{jk}$ are fixed ahead, and required to solve equations \eqref{tototo123321}. Here we only require that these equations have solutions.

\begin{proposition}
\label{monique}
Let $N,M\in\mathbb N$ and let $q_1,\ldots, q_N,p_1,\ldots ,p_M $ be two sets of distinct elements in $\mathbb H$ and let $\C_1,\ldots ,\C_N,\D_1,
\ldots ,\D_M\in\mathbb H^{n\times n}$.
Then there is a $\mathbb H^{n\times n}$-valued polynomial $T(p)=\sum_{a}p^aT_a$ such that
\[
\begin{split}
\sum_a q_j^aT_a&=\C_j,\quad j=1,\ldots, N,\\
\sum_a T_ap_k^a&=\D_k,\quad k=1,\ldots, M,
\end{split}
\]
if and only if the equations
\begin{equation}
\label{tototo123321}
q_jX_{jk}-X_{jk}p_k=\C_j-\D_k\quad j=1,\ldots, N,\quad k=1,\ldots, M
\end{equation}
are solvable.
\label{bolot}
\end{proposition}

We note that it could be that $p_j=q_k$ for some choices of indices $j$ and $k$. Then the corresponding equation \eqref{tototo123321} need not be solvable for arbitrary choices of $\C_j$ and $\D_k$.
We further remark that equations \eqref{tototo123321}  are always solvable when the $q_j$ and $p_k$ lie on different spheres.\smallskip

Last but not least, we recall that the notion of kernels with a finite number of negative squares still makes sense in the quaternionic setting; see
Remark \ref{negsq}.

\section{Proof of Theorem \ref{even}}
\setcounter{equation}{0}

Before giving the proof, we remark that, in contrast to the complex setting, the condition
\begin{equation}
\C(p)\ge0,\quad p+\overline{p}=0
\label{PR2}
\end{equation}
will not hold in general; see Example \ref{567764}.\smallskip

\begin{proof}[Proof of Theorem \ref{even}:]   It will be convenient to set, where defined, $R(x)=\chi(\C(x))$ where $x\in\mathbb R$. The main ingredients
of the proof appears in Step 3, which was proved in our previous paper \cite{MR3904447}, and in the application of Theorem \ref{ot200} in the present setting.
We first give an outline of the proof and then present the proof in a number of steps.\\

{\sl Outline of the proof:} By fixing $i\in\mathbb S$, we consider $\chi=\chi_i$, we associate with the function $\C(p)$ the rational function of a complex variable $z\in\mathbb C_i$
\begin{equation}
\label{chiR123}
R(z)=\chi(D)+\chi(C)\left(zI_{2N}-\chi(A)\right)^{-1}\chi(B)
\end{equation}
and verify that $R$ is even (Steps 1 and 2). We then verify in Steps 3 and 4 that $R(iy)\ge 0$ for real $y$ where defined. We then check in Step 5 that
$R(z)$ satisfies the symmetry:
\begin{equation}
\label{symm123}
E^{-1}\overline{R(\overline{z})}E=R(z).
\end{equation}
Using Theorem \ref{ot200}, we deduce that there is a right pseudo-spectral factorization, uniquely determined by its value, say $I_{2n}$, at infinity, such that
\[
R(z)=L_+^\sharp(z)L_+(z).
\]
We show in Steps 6, 7 and 8, that $L_+$ satisfies also the symmetry \eqref{symm123}. Uniqueness of the normalized spectral factor allows then to conclude. The result follows by inverting the map
$\chi$ and then extending from $\mathbb C_i$ to $\mathbb H$; see Step 9.\vspace{1cm}

We now go over the above strategy in a detailed manner:\\

STEP 1: {\sl The function $R(x)$ is a rational function of the real variable $x$ and satisfies
\begin{equation}
\label{batsheva}
R(-x)^*=R(x).
\end{equation}
}

To prove the claim, we restrict \eqref{PR1} to real values $p=x$ and applying the map $\chi$ we get
\[
(\chi(\C)(-x))^*=(\chi(\C))(x)
\]
 and hence \eqref{batsheva}. Furthermore, restricting \eqref{finfin} to real $x$ and applying $\chi$ now gives
\begin{equation}
\label{Rchi}
R(x)=(\chi(\C))(x)=\chi(D)+\chi(C)\left(xI_{2N}-\chi(A)\right)^{-1}\chi(B),\quad x\in \mathbb R\setminus\sigma(\chi(A)),
\end{equation}
which shows that $R(x)$ is a rational function of $x$.\\

STEP 2: {\sl $R$ admits a (uniquely defined) analytic continuation, which is even.}\\

Both sides of \eqref{chiR123} are ${\mathbb C}^{2n\times 2n}$-valued rational function of the real variable $x$, say $R(x)$, and analytic extension applied to \eqref{batsheva}
shows that $R(-\overline{z})^*=R(z)$, i.e. $R(z)$ is even.\\

STEP 3: {\sl The kernel $K_R(z,w)$ has a finite number of negative squares in $\mathbb C_r$.}\\

By taking $p=x$ and $q=y$ real and applying $\chi$ to the kernel $K_\C$ we see that $K_R(x,y)$ has a finite number of negative squares in $(0,\infty)$.
From our previous work \cite{MR3904447} we know that the meromorphic extension of $R$ to the right open half-plane, is such that $K_R(z,w)$ has a finite number of negative squares.\\

STEP 4: {\sl Where defined, $R(iy)\ge 0$ for real $y$.}\\

From Corollary \ref{cy789} we have that $R\in{\mathcal GP}$.
Since $R$ is even, we have $R(iy)=R(iy)^*$ and so $R(iy)\ge 0$.\\

STEP 5: {\sl The function $R$ satisfies the symmetry \eqref{symm123}.}\\

Indeed, for real $x$, it follows from \eqref{Rchi} that the matrix $R(x)$ is in the range of $\chi$, or, equivalently (see \cite[Lemma 3.3]{MR3904447})
\begin{equation}
E^{-1}\overline{R(x)}E=R(x),
\end{equation}
which extends analytically to
\begin{equation}
\label{symm123}
E^{-1}\overline{R(\overline{z})}E.
\end{equation}


This ends the proof of Step 5. As mentioned in the outline of the proof, there is a pseudo-spectral factorization, uniquely determined by its value, say $I_{2n}$, at infinity, such that
$R(z)=L^\sharp(z)L(z)$. Since the symmetry \eqref{symm123} is multiplicative, and since $R$ is invariant under \eqref{symm123} we have
\[
\begin{split}
R(z)&=E^{-1}\overline{R(\overline{z})}E\\
&=E^{-1}\overline{L(-z)^*}EE^{-1}\overline{L(\overline{z})}E\\
&=L^\sharp(z)L(z).
\end{split}
\]
To pursue the proof,  we now want to check that, under the normalization $L(\infty)=I_n$ it holds that
\begin{equation}
E^{-1}\overline{L(\overline{z})}E=L(z).
\label{ronit}
\end{equation}
To show that \eqref{ronit} hold, we need to show that $E^{-1}\overline{L(\overline{z})}E$ is the (uniquely defined after normalization) right pseudo-spectral factor associated with $R$ using the
minimal realization \eqref{real2} of $R(z)$. By uniqueness of the factor we will then get \eqref{ronit}.\\


For the following Step 6, see also \cite[Section 5, p. 27]{abgr1}. We use the notation
\begin{equation}
\begin{pmatrix}
\sf A&\sf B\\ \sf C& \sf D\end{pmatrix}=\begin{pmatrix}\chi(A)&\chi(B)\\ \chi(C)&\chi(D)\end{pmatrix},
\end{equation}
so that \eqref{chiR123} becomes
\begin{equation}
\label{reaR1}
R(z)={\sf D}+{\sf C} (zI_{2N}- {\sf A})^{-1}{\sf B}.
\end{equation}

STEP 6: {\sl There exists a uniquely defined invertible matrix $S\in\mathbb C^{2N\times 2N}$ such that
\begin{equation}
\begin{pmatrix}S&0\\0&I_{2n}\end{pmatrix}\begin{pmatrix}\sf  A&\sf  B\\ \sf  C&\sf  D\end{pmatrix}=
\begin{pmatrix}I_{2N}&0\\0&E^{-1}\end{pmatrix}
\begin{pmatrix}\overline{\sf  A}&\overline{\sf  B}\\ \overline{\sf  C}&\overline{\sf  D}\end{pmatrix}
\begin{pmatrix}I_{2N}&0\\0&E\end{pmatrix}
\begin{pmatrix}S&0\\0&I_{2n}\end{pmatrix}.
\label{hivers86tauee}
\end{equation}
Furthermore,
\begin{equation}
S\overline{S}=-I_{2N}
\label{hivers86taugoh}
\end{equation}
and
\begin{equation}
\overline{H}=-S^{-*}HS^{-1}.
\label{hivers86tau}
\end{equation}
}

A minimal realization of $E^{-1}\overline{R(\overline{z})}E$ is given by
\begin{equation}
E^{-1}\overline{R(\overline{z})}E=E^{-1}\overline{\sf  D}E+E^{-1}\overline{\sf  C}(zI_{2N}-\overline{\sf  A})^{-1}\overline{\sf  B}.
\label{real2}
\end{equation}
The first claim in Step 6 comes then from the uniqueness of the minimal realization and the equality \eqref{symm123}.
To prove \eqref{hivers86taugoh} we rewrite \eqref{hivers86tauee} as
\begin{equation*}
\begin{pmatrix}\overline{S}&0\\0&I_{2n}\end{pmatrix}\begin{pmatrix}\overline{\sf A}&\overline{\sf B}\\ \overline{\sf C}&\overline{\sf D}\end{pmatrix}=
\begin{pmatrix}I_{2N}&0\\0&E^{-1}\end{pmatrix}
\begin{pmatrix}{\sf A}&{\sf B}\\ {\sf C}&{\sf  D}\end{pmatrix}
\begin{pmatrix}I_{2N}&0\\0&E\end{pmatrix}
\begin{pmatrix}\overline{S}&0\\0&I_{2n}\end{pmatrix},
\end{equation*}
that is, since $E=-E^{-1}$
\begin{equation*}
  \begin{pmatrix}\overline{S}&0\\0&-I_{2n}\end{pmatrix}\begin{pmatrix}I_{2N}&0\\0&E^{-1}\end{pmatrix}
  \begin{pmatrix}\overline{\sf A}&\overline{\sf B}\\ \overline{\sf C}&\overline{\sf D}\end{pmatrix}\begin{pmatrix}I_{2N}&0\\0&E\end{pmatrix}=
\begin{pmatrix}{\sf A}&{\sf B}\\ {\sf C}&{\sf D}\end{pmatrix}
\begin{pmatrix}\overline{S}&0\\0&-I_{2n}\end{pmatrix},
\end{equation*}
which is equivalent to
\begin{equation}
    \begin{pmatrix}I_{2N}&0\\0&E^{-1}\end{pmatrix}
  \begin{pmatrix}\overline{\sf A}&\overline{\sf B}\\ \overline{\sf C}&\overline{\sf D}\end{pmatrix}\begin{pmatrix}I_{2N}&0\\0&E\end{pmatrix}\begin{pmatrix}-\overline{S}^{-1}&0\\0&I_{2n}\end{pmatrix}=
\begin{pmatrix}-\overline{S}^{-1}&0\\0&-I_{2n}\end{pmatrix}\begin{pmatrix}{\sf A}&{\sf B}\\ {\sf C}&{\sf D}\end{pmatrix},
\label{hivers86taueeyarkon}
\end{equation}
and hence the result by uniqueness of $S$.\smallskip

To prove \eqref{hivers86tau} we use the uniqueness of $H$ satisfying
the equations
\begin{eqnarray}
\label{rubeole1}
H{\sf A}&=&-{\sf A}^*H\\
H{\sf B}&=&{\sf C}^*.
\label{rubeole2}
\end{eqnarray}
From \eqref{rubeole1} and \eqref{hivers86tauee} we can write
\[
\begin{split}
\overline{H}\overline{\sf A}&=-\overline{\sf A^*}\overline{H}\\
&\iff\\
\overline{H}S{\sf A}S^{-1}&=-S^{-*}{\sf A}^*S^{*}\overline{H}\\
&\iff\\
S^*\overline{H}S{\sf A}&=-{\sf A}^*S^{*}\overline{H}S.
\end{split}
\]

Next, using $E^{-1}=-E$ and $\overline{S}=-S$ we have
\[
\overline{\sf C}=E^{-1}{\sf C}\overline{S}=E{\sf C}S^{-1},
\]
and so starting from \eqref{rubeole2},
\[
\begin{split}
\overline{H}\overline{\sf B}&=\overline{\sf C}^*\\
&\iff\\
\overline{H}\overline{\sf B}E&=\overline{\sf C}^*E\\
&\iff\\
\overline{H}SB&=S^{-*}{\sf C}^*E^*E\\
&\iff\\
S^*\overline{H}S{\sf B}&={\sf C}^*.
\end{split}
\]
So, both $H$ and $-S^*HS$ satisfy the equations \eqref{rubeole1}-\eqref{rubeole2} characterizing $H$, and hence
\eqref{hivers86tau}.\\

At this stage we consider the two minimal realizations \eqref{reaR1} and \eqref{real2} of $R(z)$ and show, using formula \eqref{corona123}, that they lead to the same pseudo-spectral factors.\\

STEP 7: {\sl Let $\mathcal M_\pm$ and $\mathcal M^\times_\pm$ be the four
subspaces of $\mathbb C^{2N}$ defined as in the statement of Theorem \ref{ot200} associated with the realization \eqref{reaR1} (i.e. \eqref{chiR123}).
Then the corresponding spaces associated with the realization \eqref{real2} are $S\mathcal M_\pm$ and $S\mathcal M_\pm^\prime$.}\\

Since $\overline{\sf A}=S{\sf A}S^{-1}$ the spaces $M_\pm$ are invariant under conjugation
We have
\[
\begin{split}
  \overline{\sf A}S\mathcal M_+&=S{\sf A}S^{-1}S\subset \mathcal M_+\\
  &=S{\sf A}\mathcal M_+\\
  &\subset S\mathcal M_+\quad\mbox{\rm since ${\sf A}\mathcal M_+\subset \mathcal M_+$}
\end{split}
\]
and $\sigma(\overline{\sf A}|_{\mathcal M_+})\subset\mathbb C_r$
and similarly for $\mathcal M_-$.\\

Note that
\[
\overline{\sf A^\times}=\overline{\sf A}-\overline{\sf B}\overline{\sf C}=S{\sf A}S^{-1}-S{\sf B}E^{-1}E{\sf C}S^{-1}=S{\sf A}^\times S^{-1}
\]
So  $\overline{\sf A^\times}S \mathcal M^\times_\pm\subset S\mathcal M^\times_\pm$ and $\sigma(\overline{{\sf A}^\times}|_{\mathcal M^\times_+})\subset\mathbb C_r$ and
$\sigma(\overline{\sf A^\times}|_{\mathcal M^\times_-})\subset\mathbb C_\ell$.\\

Since we have the matching conditions \eqref{mathcond} we get
\begin{equation}
\label{sum1}
\mathbb C^{2N}=S\mathcal M_-\stackrel{\cdot}{+}S\mathcal M_+^\times
\end{equation}

and

\begin{equation}
\label{sum1}
\mathbb C^{2N}=S\mathcal M_+\stackrel{\cdot}{+}S\mathcal M_-^\times
\end{equation}

We now check condition ${\rm (v)}$ from Theorem \ref{ot200}. The first condition is
\begin{equation}
\label{567ui}
\overline{H}(S\mathcal M_+)=(S\mathcal M_+)^{\perp}
\end{equation}
i.e.
\[
\langle \overline{H}Sm_+,Sn_+\rangle_{\mathbb C^{2N}}=0,\quad m_+,n_+\in\mathcal  M_+
\]
But this can be rewritten as
\[
\langle S^*\overline{H}Sm_+,n_+\rangle=0
\]
which holds since $S^*\overline{H}S=H$ and $H\mathcal M_+=\mathcal M_+^{\perp}$. The other claims in ${\rm (v)}$ are proved in the same way.\\


STEP 8: {\sl The spectral factors associated to the two realizations of $R$ coincide.}\\

Let $\Pi_+$ be the projection onto $\mathcal M_+^\times$ along $\mathcal M_-$. Then, the projection onto $S\mathcal M_+^\times$ along $S\mathcal M_-$ is $S\Pi_+S^{-1}$.
Indeed, $P=S\Pi_+S^{-1}$ is a projection since
\[
P^2=S\Pi_+S^{-1}S\Pi_+S^{-1}=S\Pi_+^2S^{-1}=P
\]
with range $S({\rm ran}\,\Pi_+)$ and kernel $S({\ker}\,\Pi_+)$. So using formula \eqref{corona123}
the pseudo-spectral factor associated with the realization
\[
R(z)=I_{2n}+E^{-1}\overline{\sf C}(zI_{2N}-\overline{\sf A})^{-1}\overline{\sf B}E,
\]
with value $I_{2n}$ at infinity is
\[
\begin{split}
  L_1(z)&=I_{2n}+E^{-1}\overline{\sf C}(S\Pi_+S^{-1})(zI_N-\overline{\sf A})^{-1}\overline{\sf B}E\\
  &=I_{2n}+{\sf C} S^{-1}(S\Pi_+S^{-1})(zI_{2N}-S{\sf A}S^{-1})^{-1}S\sf B\\
    &=I_{2n}+{\sf C} \Pi_+(zI_{2N}-{\sf A})\sf B\\
  &=L_+(z),
\end{split}
\]
which is formula \eqref{corona123} for the right spectral factor, and hence, by uniqueness of the factor, this concludes the proof of the step.\\

STEP 9: {\sl The factorization \eqref{ronron} holds.}\\

Let
\[
L(x)=\chi(M(x))
\]
where $M$ is $\mathbb H^{n\times n}$-valued.
We have
\[
\chi(\C(x))=R(x)=\chi(M(-x)^*)\chi(M(x))
\]
and the result follows by inverting $\chi$ and then extending first $x$ to $z\in\mathbb C_i$ and then from $\mathbb C_i$ to $\mathbb H$. We note that such an extension obviously exists since we can let $i\in\mathbb S$ vary and it is unique by the Identity Principle, see \cite{zbMATH06658818}.
\end{proof}

\begin{corollary}
Let $\C$ be as in Theorem \ref{even}.
Then, there is an $\mathbb H^{n\times n}$-valued slice-hyperholomorphic rational function $L_-$, without poles and zeros in the left right half-space, uniquely determined by the condition
$L_-(\infty)=I_n$,
and such that
\begin{equation}
\label{ronron}
\C(p)=L_-^\sharp(p)\star L_-(p).
\end{equation}
\label{even123321}
\end{corollary}

The same limiting process as in the proof of Corollary \ref{cycy} gives:

\begin{corollary}
Let $\C\in\mathcal{GPE}(\mathbb H)$ be analytic at infinity. Then there exist factorizations $\C(p)=L_+^\sharp(p)\star L_+(p)=
L_-^\sharp(p)\star L_-(p)$, where the poles and zeros of $L_+$ (resp. $L_-$) are in
the closed left half-space (resp. the closed right half-space). When $\C$ is a polynomial so are the factors $L_\pm(z).$
\end{corollary}

\begin{proof}
The proof goes as in Corollary \ref{cycy}, but we need first to check that $\C(\infty)\ge 0$. To that purpose we use the map $\chi$ and taking the limit as $x\in\mathbb R$ goes to infinity.
\end{proof}
\section{Interpolation}
\setcounter{equation}{0}
In this section we consider some interpolation problems in the class of quaternionic rational generalized positive functions. We begin with
the following factorization result, counterpart of Corollary \ref{cor567}.

\begin{proposition}
Let $L_1,\ldots, L_U$ be rational $\mathbb H^{n\times n}$-valued functions analytic at infinity and assume that
\begin{equation}
\sum_{u=1}^U(L_u(\infty))^*L_u(\infty)>0.
\end{equation}
Then, there exists a $\mathbb H^{n\times n}$-valued rational function analytic at infinity and such that
\begin{equation}
\sum_{u=1}^UL_u^\sharp(p)\star L_u(p)=L^\sharp(p)\star L(p)
\end{equation}
\end{proposition}

We begin with a lemma; note that, in the lemma we allow for two points to be symmetric, i.e. it can be that $p_j+\overline{p_u}=0$.
\begin{lemma}
Let $N,M\in\mathbb N$ and let $q_1,\ldots, q_N,r_1,\ldots ,r_M $ be pairwise different points in $\mathbb H$. Assume that no three of the points
$q_1,\ldots, q_N$ and $-\overline{r_1},\ldots ,-\overline{r_M}$
are on the same sphere, and let $\C_1,\ldots ,\C_N,\D_1,
\ldots ,\D_M\in\mathbb H^{n\times n}$.
There exists a polynomial $T(p)$ such that
\begin{equation}
T(q_j)=\C_j,\quad j=1,\ldots, N
\end{equation}
and
\begin{equation}
T^\sharp(r_k)=\D_k,\quad j=1,\ldots, M.
\label{conditionc}
\end{equation}
\label{tali}
\end{lemma}

\begin{proof}
We look for  $T$ of the form (with $N\in\mathbb N$ to be determined)
\[
T(p)=\sum_{a=0}^{N}p^aT_a.
\]
Conditions \eqref{conditionc} can be rewritten as
\[
\sum_{a=0}^N(-r_k)^aT_a^*=\Psi_k,\quad k=1,\ldots, N,
\]
that is, taking conjugate,
\begin{equation}
\label{theotherside}
\sum_{a=0}^NT_a(-\overline{r_k})^a=\Psi_k^*,\quad k=1,\ldots, N .
\end{equation}
The system of conditions in the lemma  corresponds thus to a two-sided interpolation problem of the kind considered by Vladimir Bolotnikov in \cite{bolotnikov2019lagrange}, without the compatibility
conditions being pre-assigned. A solution will then always exist since no three of the points $q_1,\ldots, q_N$ and $-\overline{r_1},\ldots ,-\overline{r_M}$ are on the same sphere, and hence the compatibility
conditions \cite[(3.2)]{bolotnikov2019lagrange} have always a solution.
\end{proof}

We consider the following problem:

\begin{problem}
\label{vartyu}
Given $p_1,\ldots, p_N\in\mathbb H$, not three of them on a common sphere,
and $\C_1,\ldots, \C_N\in\mathbb H^{n\times n}$,  find an even generalized positive $\mathbb H^{n\times n}$-valued function, slice hyperholomorphic in
$\mathbb H_+$ and such that
\begin{equation}
\C(p_j)=\C_j,\quad j=1,\ldots, N.
\end{equation}
\end{problem}

We first consider the case
\begin{equation}
p_u+\overline{p_v}\not=0,\quad \forall u,v\in\left\{1,\ldots, N\right\}
\label{inforce!!!}.
\end{equation}
In particular no points are purely imaginary. The condition holds in particular when
all the points are in the open right half-space.\\

\begin{proposition}
Assume \eqref{inforce!!!}.
Then Problem \ref{vartyu} has a solution.
\end{proposition}

\begin{proof}
Since $p_u$ and $-\overline{p_u}$ do not belong to the same sphere, using Lemma \ref{tali} we build for $u=1,\ldots, N$  a rational slice hyperholomorphic ${\mathbb H}^{n\times n}$-valued function
$L_u$ such that
\[
  L_u(p_j)=\begin{cases}\,\, 0_{n\times n},\quad j\not =u\\
    \,\, I_n,\hspace{4mm}\quad j=u,\end{cases}
\]
and
\[
L_u^\sharp(p_u)=\C_u.
\]
Then, \eqref{zebda} gives
\[
(L_u\star L_u^\sharp)(p_u)=L_u^\sharp(p_u)=\C_u,
\]
and
\[
\Phi(p)=  \sum_{u=1}^N
L_u(p)\star L_u^\sharp(p)
\]
answers the question.
\end{proof}

When  one has symmetric points, or purely imaginary points, among the interpolation nodes, that is $p_u+\overline{p_v}=0$ for some $u,v\in\left\{1,\ldots, N\right\}$
the above procedure needs to be adapted. One adds supplementary compatibility conditions. More precisely, we look for $L_u$ such that

\begin{equation}
L_u(p_j)=\begin{cases}\,\, 0_{n\times n},\quad j\not =u\\
 \,\, I_n,\hspace{4mm}\quad j=u,\end{cases}
\end{equation}

and
\[
L_u^\sharp(p_u)=\C_u .
\]
Let $L(p)=\sum_{k}p^kL_k$, with $L_k\in\mathbb H^{n\times n}$. The condition $L_u^\sharp(p_u)=\C_u$, i.e.
\[
\sum_{k}(-p_u)^kL_k=\C_u
\]
becomes after taking conjugate
\begin{equation}
\label{deb}
\sum_{k}L_k^*(-\overline{p_u})^k=\C_u^*.
\end{equation}
Equation \eqref{deb} is a right-sided interpolation condition, and following the results in \cite{bolotnikov2019lagrange} (see Proposition \ref{monique})
we need to solve the corresponding equation \eqref{tototo123321}, i.e.
\[
p_uX_{uv}+X_{uv}\overline{p_v}=I_n-\C_v^*.,\quad i.e.  \quad p_uX_{uv}-X_{uv}p_u=I_n-\C_v^*.
\]
Then,  using \eqref{zebda} we have
\[
(L\star L^\sharp)(p_u)=\sum_{v=0}^Np_u^vL_v=\C_u.
\]

\begin{remark}{\rm This paper dealt with  an aspect of quaternionic linear system theory. We refer to the papers \cite{MR733953,pereira2006quaternionic,MR2131920}
for other studies in this field, and to the papers \cite{acls_milan,acs1,acs-survey}
and books \cite{zbMATH06658818,zbMATH07183320}
for an approach to quaternionic linear system theory in connection with Schur analysis.\smallskip

In the arguments in this paper the uniqueness of the spectral factor played a key role.
We here remark that extending the above interpolation approach beyond the case of even functions, is out of the scope of this work, and
requires different methods. There are still whole sections of quaternionic linear system theory which remain to be developed.}
\end{remark}

\end{document}